\documentclass[12pt,reqno]{amsart}

\usepackage{amsmath,amsfonts,amsthm,amssymb,amsxtra,bbm,multirow}

\usepackage{palatino}

\usepackage{graphicx,enumerate}
\usepackage{hyperref}
\usepackage{pdfsync}
\usepackage{color}
\usepackage[T1]{fontenc} 
\usepackage{lmodern}

\newtheorem{theorem}{Theorem}
\newtheorem{lemma}[theorem]{Lemma}
\newtheorem{corollary}[theorem]{Corollary}

\theoremstyle{remark}

\renewcommand{\epsilon}{\varepsilon}

\newcommand{\eps}{\varepsilon}
\newcommand{\E}{\mathcal{E}}

\renewcommand{\P}{\mathcal{P}}
\newcommand{\R}{\mathbb{R}}

\DeclareMathOperator{\diam}{diam}
\DeclareMathOperator{\dist}{dist}

\DeclareMathOperator{\supp}{supp}

\title[Strong Attraction Limit of Nonlocal Interaction Energies]{On the Strong Attraction Limit for a Class of Nonlocal Interaction Energies}

\author{Almut Burchard}
\address{Dept. of Mathematics, University of Toronto, Toronto, ON, Canada}
\email{almut@math.toronto.edu}

\author{Rustum Choksi}
\address{Dept. of Mathematics and Statistics,
		McGill University, Montr\'{e}al, QC, Canada}
\email{rustum.choksi@mcgill.ca}

\author{Elias Hess-Childs}
\address{Dept. of Mathematics and Statistics, McGill 
University, Montr\'{e}al, QC, Canada\newline
\indent
{Current address:
Courant Institute of Mathematical Sciences
New York, NY}}
\email{elias.hess-childs@courant.nyu.edu}

\date{February 25, 2020}                                        

\begin{document}
\maketitle

\begin{abstract} This note concerns the problem of
minimizing a certain family of non-local energy functionals 
over measures on $\R^n$, subject to a mass constraint,
in a strong attraction limit.
In these problems, the total energy is an 
integral over pair interactions
of attractive-repulsive type. The interaction 
kernel is a sum of competing power law 
potentials with attractive powers 
$\alpha \in (0, \infty)$ and repulsive powers 
associated with Riesz potentials.   
The strong attraction limit
$\alpha \to \infty$ is addressed via Gamma-convergence, 
and 
minimizers of the limit are characterized in terms of 
an isodiametric capacity problem.
We also provide evidence for symmetry-breaking
in high dimensions.
\end{abstract}

\maketitle

\section{Introduction and Statement of the Results} 

We consider mass-constrained variational problems of the form
\begin{equation}\label{prob-P}
\begin{cases} {\rm Minimize} \quad  
{\displaystyle
\E_{\alpha,\lambda}(\mu)\, :=\, \int_{\R^n}\int_{\R^n} 
K_{\alpha,\lambda}{(x-y)} \, d\mu(x)d\mu(y) }\\[.5em]
{\rm over} \,\, \P:=\{\mu\ \text{Borel measure on}\ 
\R^n:\mu(\R^n)=1\} \,,
\end{cases}
\end{equation} 
where the interaction kernel  is given by
\begin{equation}\label{kernel}
K_{\alpha,\lambda} (x-y) \, :=\, {|x-y|^\alpha} \, 
+\,  {|x-y|^{-\lambda}} \qquad 
\text{\rm with $\alpha \in (0, \infty)$ and $\lambda \in (0,n)$.} 
\end{equation}
These kernels are strongly repulsive at short range, 
with the repulsion controlled by the exponent $\lambda$, 
and attractive at long range, with the attraction 
controlled by $\alpha$, see Fig.~\ref{fig-ker}.  
Since the kernels are lower semicontinuous, locally
integrable, and grow at infinity,
by the results of~\cite{SST, CCP}, Problem~\eqref{prob-P}
has a  global minimizer.
\begin{figure}[h]
  \includegraphics[width=0.3\textwidth]{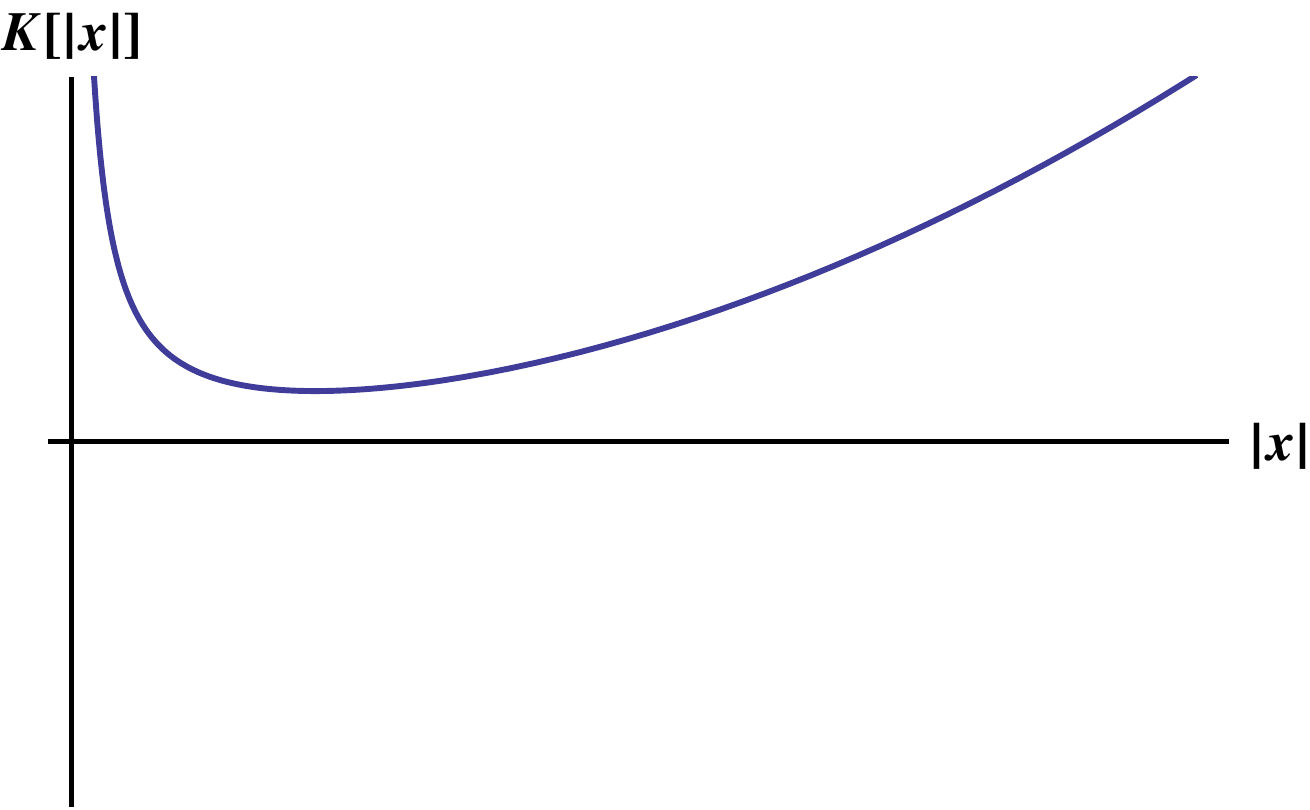}
  	\caption{\small Shape of the interaction kernel 
$K_{\alpha,\lambda} (|\cdot|)$ for 
{$\alpha\in (0,\infty)$ and
$\lambda\in(0,n)$.}} 
\label{fig-ker}
\end{figure}

Variational problems of the form~\eqref{prob-P} arise 
in connection with a class of models for aggregation 
and self-assembly that
have recently received much attention 
(see for example,~\cite{BT} and the references therein). 
In those models, a population density $\rho$ evolves
according to the equation
\[	
\rho_t + \nabla \cdot (\rho \mathbf{v}) = 0\,, \qquad 
\mathbf{v} = - \nabla K_{\alpha,\lambda} \ast \rho\,,
\]
which is the gradient flow of the 
energy $\E_{\alpha,\lambda}(\mu)$ on absolutely continuous
measures $\mu=\rho\,dx$
in the 2-Wasserstein metric (cf.~\cite{CaDiFiLaSl}).
Energy minimizers represent
stable steady-states of the aggregation
process.

Here, we study the minimization problem~\eqref{prob-P}
in the strong attraction regime where $\alpha \to \infty$.
In this limit, finite energy alone restricts the support of 
a measure to have diameter no larger than one.

In Fig.~\ref{fig-PS}, we present a few particle simulations 
in dimension $n=2$ which suggest that as $\alpha$ increases, 
minimizers concentrate on the boundary of the ball of diameter 1
for some values of $\lambda$; but spread
out (non-uniformly) over the ball for larger values of $\lambda$.
A broader range of behaviour is expected  for other parameters and in higher
dimensions (see for example Fig.~\ref{fig-3D}).

\begin{figure}[h]
 \includegraphics[width=0.26\textwidth]{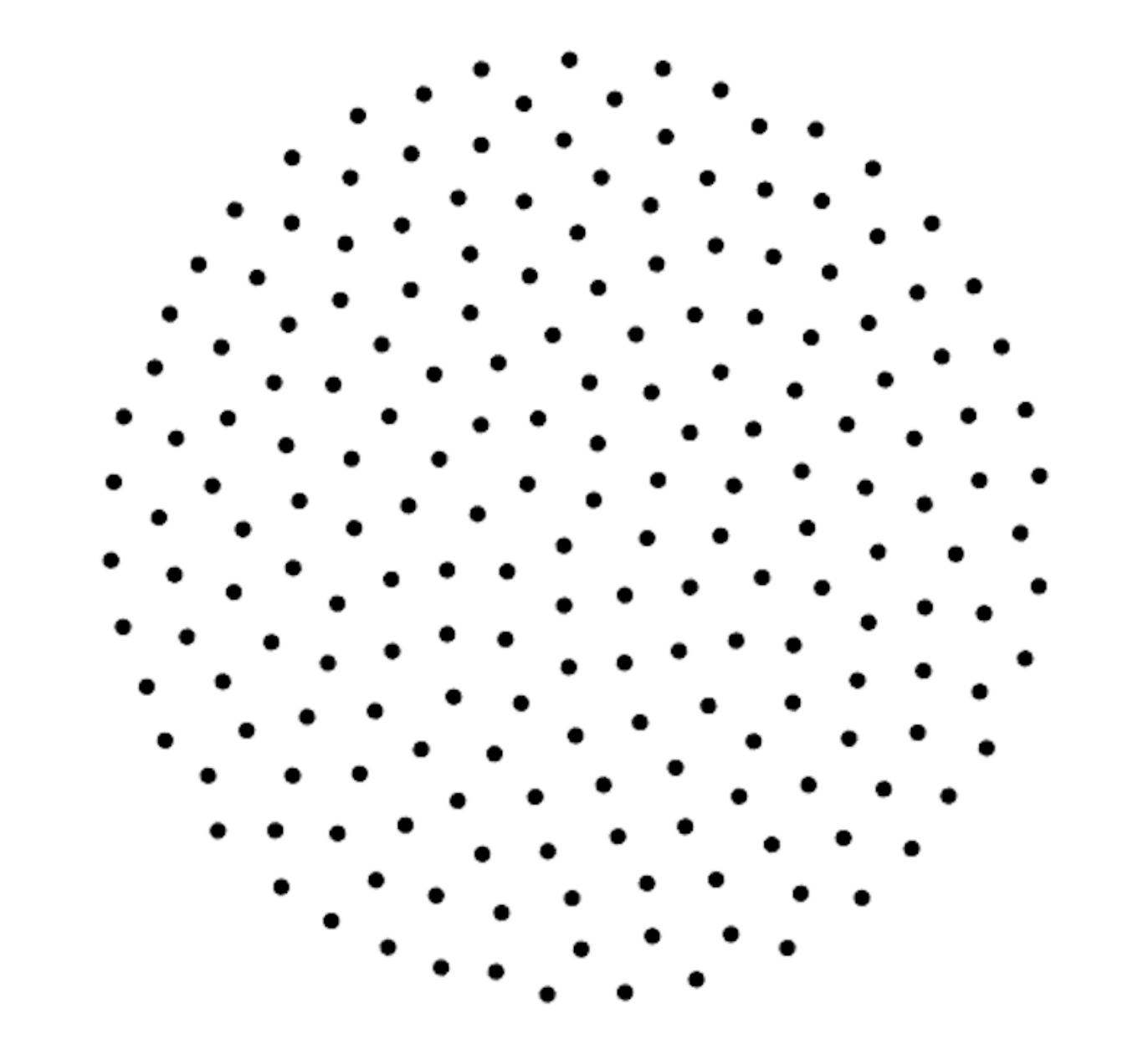} \qquad   \includegraphics[width=0.25\textwidth]{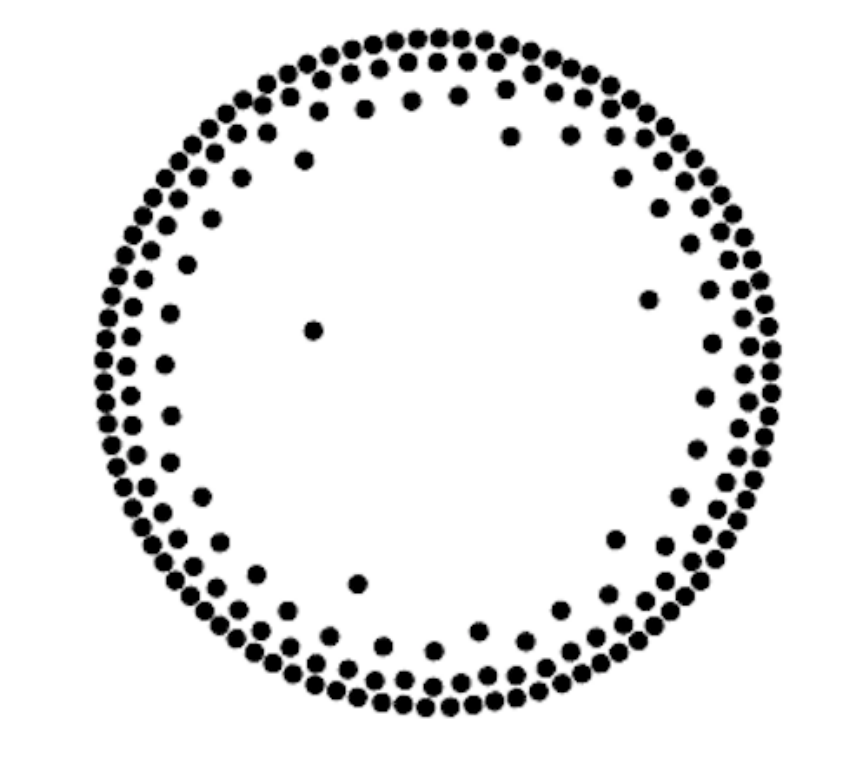}\qquad 
    \includegraphics[width=0.27\textwidth]{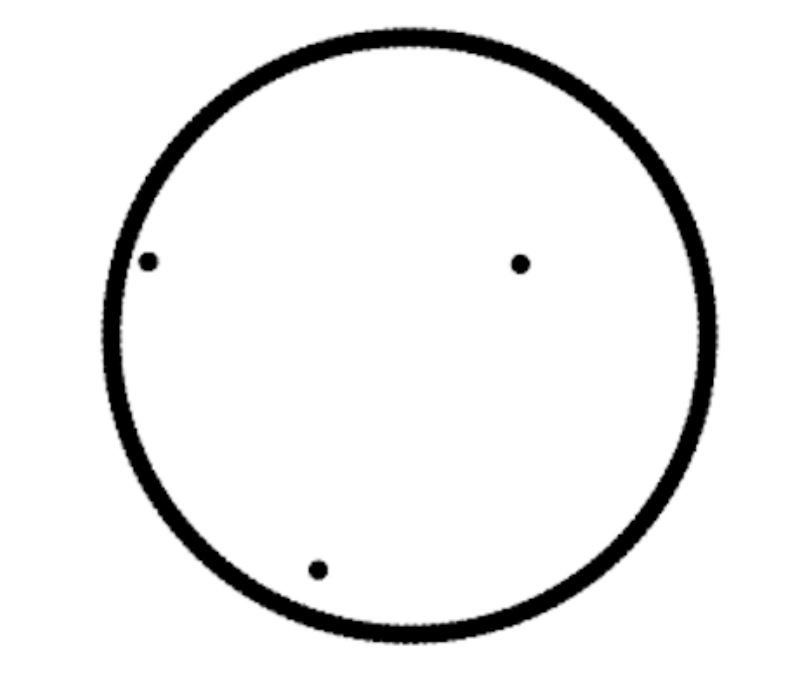} 
      \qquad \quad \qquad \quad \qquad \quad\qquad \quad
  \includegraphics[width=0.25\textwidth]{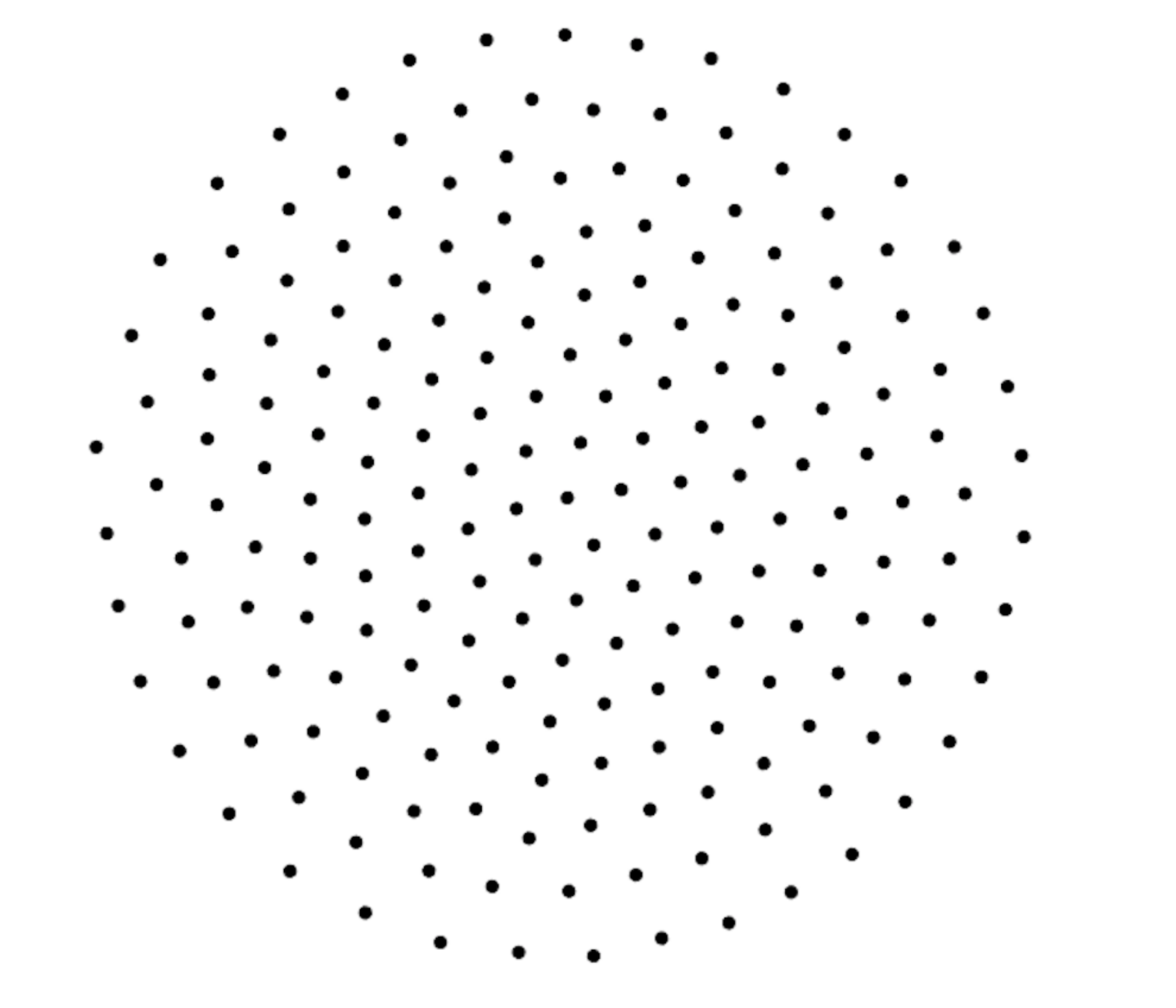}  \qquad \includegraphics[width=0.26\textwidth]{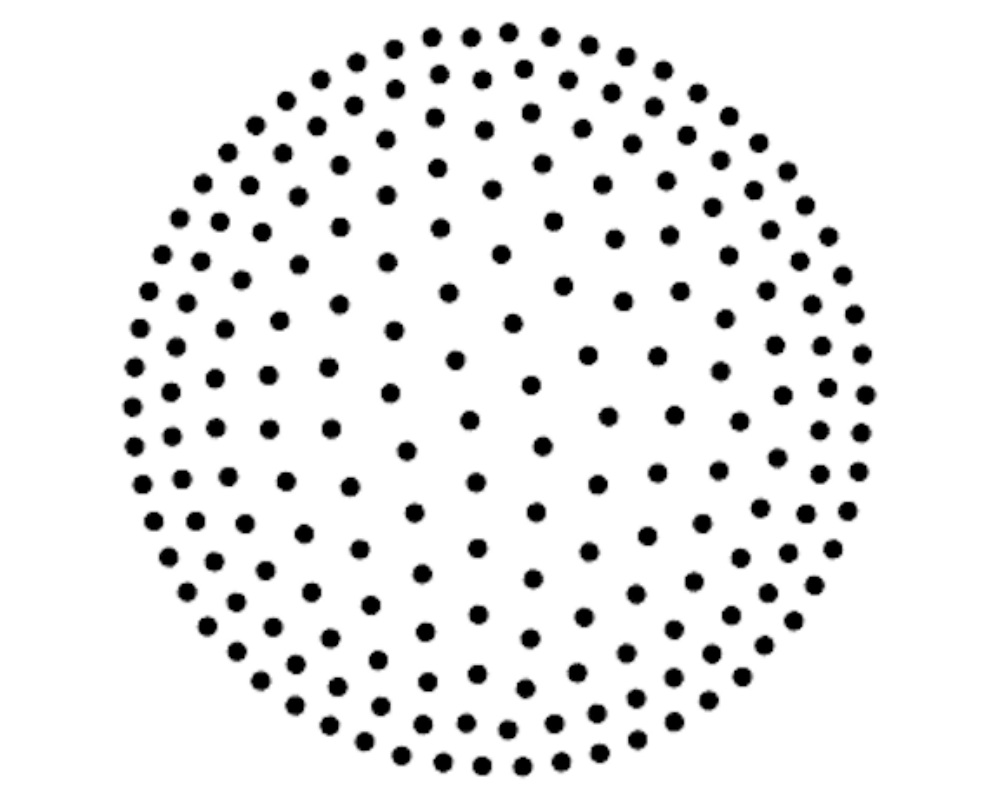}
  \qquad \includegraphics[width=0.26\textwidth]{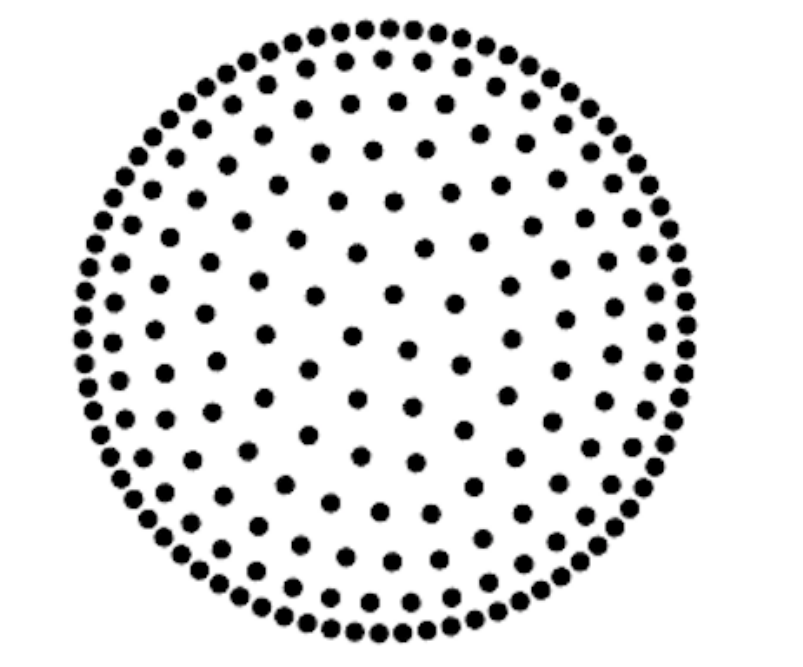}  \caption{\small Particle simulations associated with minimizers of 
(\ref{prob-P}) in dimension $n=2$.
Each particle $i=1,\dots, N$ is tracked
via the system of ODEs 
\[ \frac{dX_i}{dt} \, = \, -\frac{1}{N} 
\sum_{j=1}^N \nabla K_{\alpha,\lambda} (X_i - X_j)
\] 
until the configuration stabilizes. The interaction
kernel is give by Eq.~\eqref{kernel},
where the exponent of attraction
ranges through $\alpha=2, 20, 200$ (from left to right).
{\em Top row:}\  Repulsive term replaced 
with the logarithmic term $-\log|x-y|$ that corresponds
to the Newton potential ($\lambda=n-2$)
in two dimensions.
{\em Bottom row: }\ Exponent of repulsion $\lambda = 1$,
which lies in the super-Newtonian regime.} 
\label{fig-PS}
\end{figure}

Our first result is that in the limit as $\alpha\to\infty$,
Problem~\eqref{prob-P} approaches 
{the problem of minimizing
\begin{equation}\label{prob-L}
\E_{\infty,\lambda}(\mu)\, :=\, 
\begin{cases}
{\displaystyle
\int_{\R^n}\int_{\R^n} 
|x-y|^{-\lambda} \,d\mu(x)d\mu(y)} 
\quad & \text{if}\diam(\supp\mu) \leq1\\[.5em]
+\infty & \text{otherwise}
\end{cases}
\end{equation} 
over $\P$.} The limit is understood 
in the sense of Gamma-convergence.

\begin{theorem}[Strong attraction limit]
\label{thm:limit} 
{Let $\lambda\in (0,n)$.}
Then $\E_{\alpha,\lambda}\xrightarrow{\Gamma} 
\E_{\infty,\lambda}$ as $\alpha \to \infty$ 
in the weak topology of measures.
\end{theorem}

The limiting problem admits a solution:

\begin{theorem}[Existence]
\label{thm:minimizer}
The functional
$\E_{\infty,\lambda}$ has a global minimizer in $\P$.
\end{theorem}
The proofs of Theorems~\ref{thm:limit} and~\ref{thm:minimizer}
are presented in  Section \ref{sec-converge}.

\medskip\noindent{\bf Remark.}
In the literature, the interaction
kernel is sometimes normalized to
\begin{equation}\label{kernel1}
\tilde K_{\alpha,\lambda} 
{(x-y)} \, :=\, \frac{1}{\alpha}
|x-y|^{\alpha} + \frac{1}{\lambda} |x-y|^{-\lambda} \,,
\end{equation}
which assumes its minimum when $|x-y|=1$~(cf.~\cite{BCT}). 
This normalization
can be achieved by acting on $\P$ with a suitable dilation.
For the normalized kernel, the conclusions of Theorem~\ref{thm:limit}
hold with $\frac{1}{\lambda} \E_{\infty,\lambda}$ as the limiting
functional, and Theorem~\ref{thm:minimizer} applies without change.

\medskip
We then consider the nature of minimizers for
the limiting problem $\E_{\infty,\lambda}$.
This turns out to be a rather subtle 
question; indeed,  
{due to the diameter constraint,
the functional $\E_{\infty,\lambda}$ is 
non-convex on $\P$. Our approach is to rephrase}
the limiting problem as an isodiametric capacity
problem.  
More precisely, for $\lambda\in (0,n)$ and a set $A\subset\R^n$,  we define the  {\it $\lambda$-capacity} of  $A$ to be 
\begin{equation}\label{lam-cap}
C_\lambda (A)\, : =\, 
\Bigl(\inf_{\nu\in \P}
\bigl\{I_\lambda(\nu)\ \big\vert\ 
\supp\nu\subset A\bigr\}\Bigr)^{-1}, 
\end{equation} 
where 
\[ I_{\lambda}(\nu)\,
:= \int_{\R^n}\int_{\R^n}  |x-y|^{-\lambda}
{\,d\nu(x)d\nu(y)}. \] 
In the special case where $n=3$ and $\lambda=1$,
$C_\lambda$ agrees (up to a multiplicative constant)
with the electrostatic capacity of $A$. It is straightforward (cf. Lemma \ref{lem:isodiametric}) to show that 
$$
\inf_{\nu\in \P} \E_{\infty,\lambda} (\nu)
= \Bigl( \sup_{A\subset\R^n}\bigl\{ C_\lambda(A)\ \big\vert\ 
\diam (A) \le 1\bigr\}\Bigr)^{-1},
$$
with a direct relationship between the optimal measure $\nu$ on the left and optimal set $A$ on the right. 
This allows us to exploit tools from 
potential theory (cf.~\cite{L}) to partially characterize 
the support of minimizers.

 \begin{theorem}
[{Properties of minimizers of the limit problem}]
\label{thm:support} Let  $n \ge 3$,
$\lambda\in (0,n)$, and 
assume that $\mu$ minimizes $\E_{\infty,\lambda}$
on $\P$.  Then there exists a convex body
$W$ of constant width 1 such that
$$ 
\begin{cases}
\supp\mu \subset\partial W\,, \qquad& \lambda\in (0, n-2)
\ \text{\rm (sub-Newtonian),}\\
\supp\mu = \partial W\,, & \lambda= n-2 \ \text{\rm (Newtonian),}\\
\supp\mu = W\,, \qquad& \lambda\in (n-2,n)\ \text{\rm (super-Newtonian).}
\end{cases}
$$
\end{theorem}
The set $W$ may depend on $\mu$ as well
as $\lambda$. We do not know whether minimizers are unique 
up to translation, and whether $\E_{\infty,\lambda}$
admits additional critical points, including local minima.
The proof of Theorem~\ref{thm:support} 
is presented in Section \ref{sec-characterization}. 

The theorem extends to lower dimensions as follows.
For $n=1$, the entire range $\lambda\in (0,1)$
is super-Newtonian, and the support of any minimizing
measure is an interval of length one.
In dimension $n=2$, the entire range $\lambda\in (0,2)$
is super-Newtonian as well, and the support of any minimizing
measure is a planar convex set $W$ of constant width 1.
The role of the Newton potential
$|x-y|^{2-n}$ is played by the logarithmic kernel
$-\log|x-y|$; in this case, the support of a minimizer
is the boundary of a planar convex set of 
constant width 1.

In Section \ref{sec-cap-est}, we  prove the following result pertaining to the asymmetry of minimizers in high space 
dimensions. Precisely, we prove 
\begin{theorem}[Asymmetry of minimizers in high dimensions]
\label{thm:sym-break} For every $\lambda>0$ there 
exists $N$ such that for all $n\geq N$,
$$
\sup \left\{ C_\lambda(A) \ \big\vert\ 
A\subset \R^n, \diam(A)\le 1\right\}  > C_\lambda({B^{(n)}_{1/2}})\,,
$$
where $B^{(n)}_{1/2}=\{x\in\R^n:|x|\leq\frac{1}{2}\}$. 
\end{theorem}
This demonstrates that for any fixed value of $\lambda>0$, 
the ball ceases to be optimal when $n$ is sufficiently large. 
Thus, in this 
regime optimal measures are supported 
on the boundary of sets that are not 
radially symmetric. As a result, minimizers 
of $\E_{\alpha,\lambda}$ in high dimensions
must also be asymmetric when $\alpha$ is large.

The prospect of symmetry-breaking 
presents an interesting, largely open, question.  
Even in low space dimensions, we suspect that when $0<\lambda<< n-2$ the
maximal capacity among bodies of given diameter
may be achieved by non-symmetric sets,
and that the equilibrium measure may be supported 
on a proper 
subset of the boundary. For example, in Fig.~\ref{fig-3D} we 
present the results of 3D particle simulations for 
$\lambda = 0.01$ and respectively, $\alpha = 2, 20, 200$. 
The simulations suggest that minimizers are asymmetric for 
large $\alpha$.  However, the number of particles is too small
to draw conclusions about the supports of minimizing measures.
\begin{figure}[h]
  \includegraphics[width=0.3\textwidth]{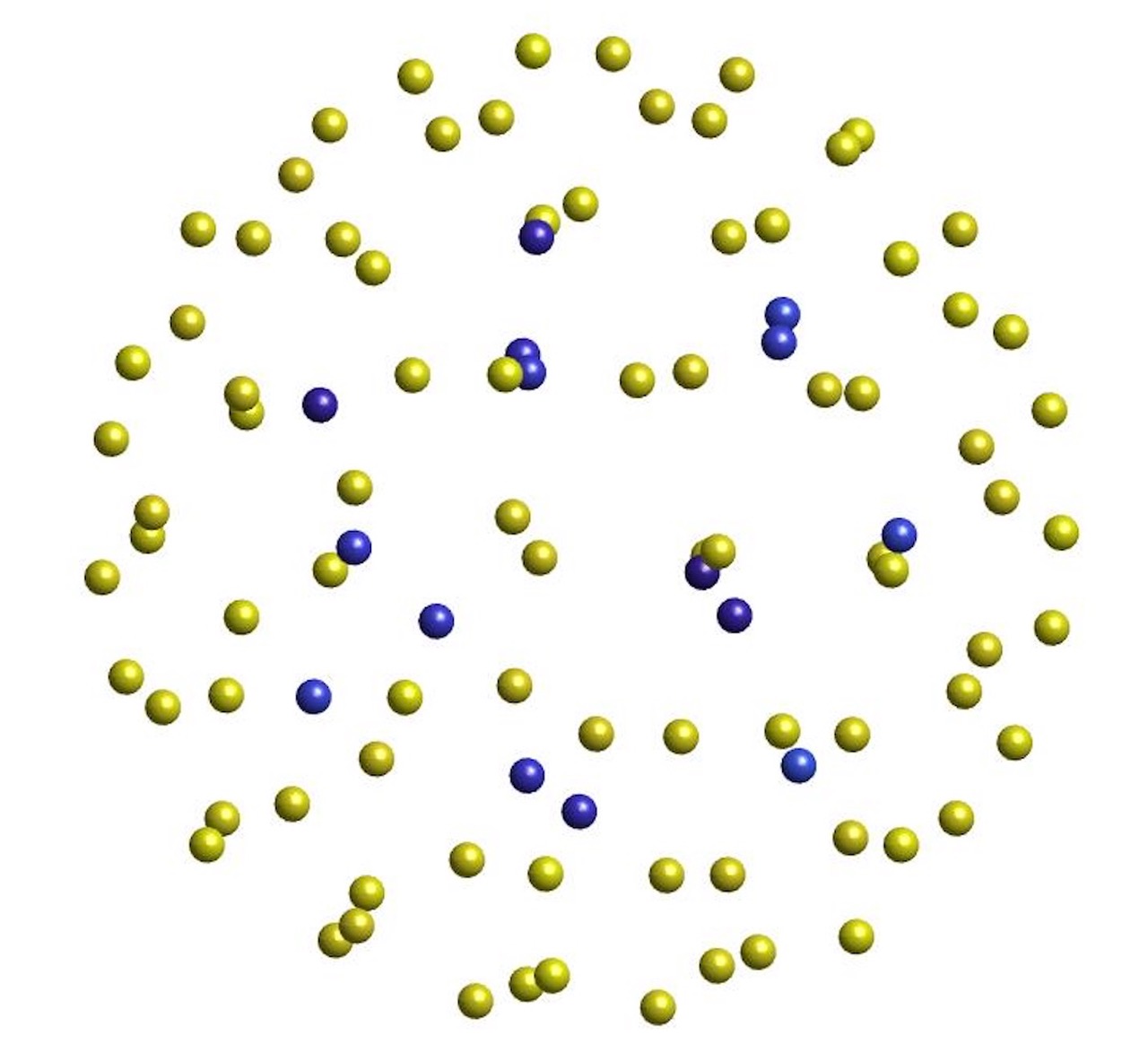}\,\,\,\,\,
    \includegraphics[width=0.3\textwidth]{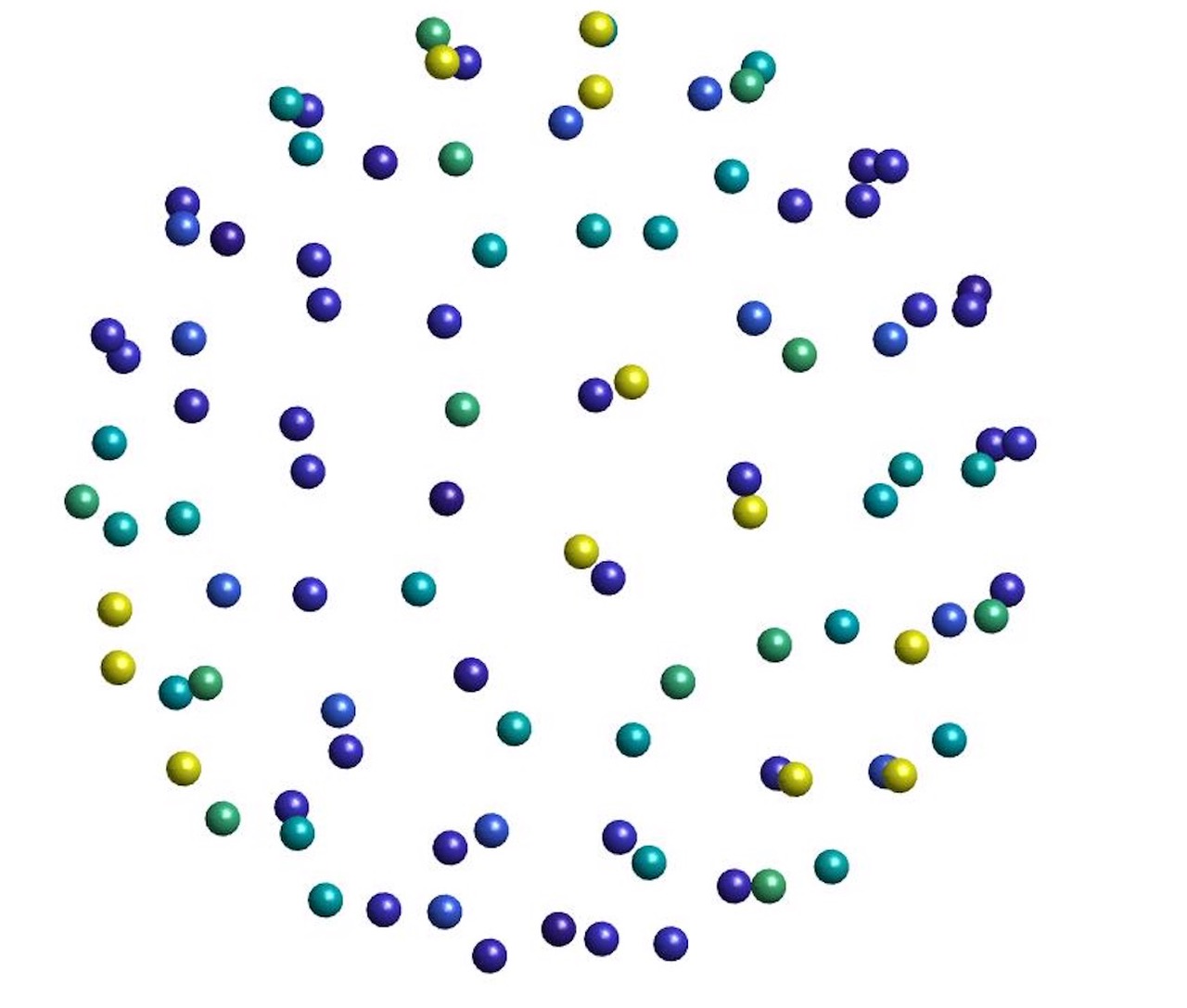}
      \includegraphics[width=0.3\textwidth]{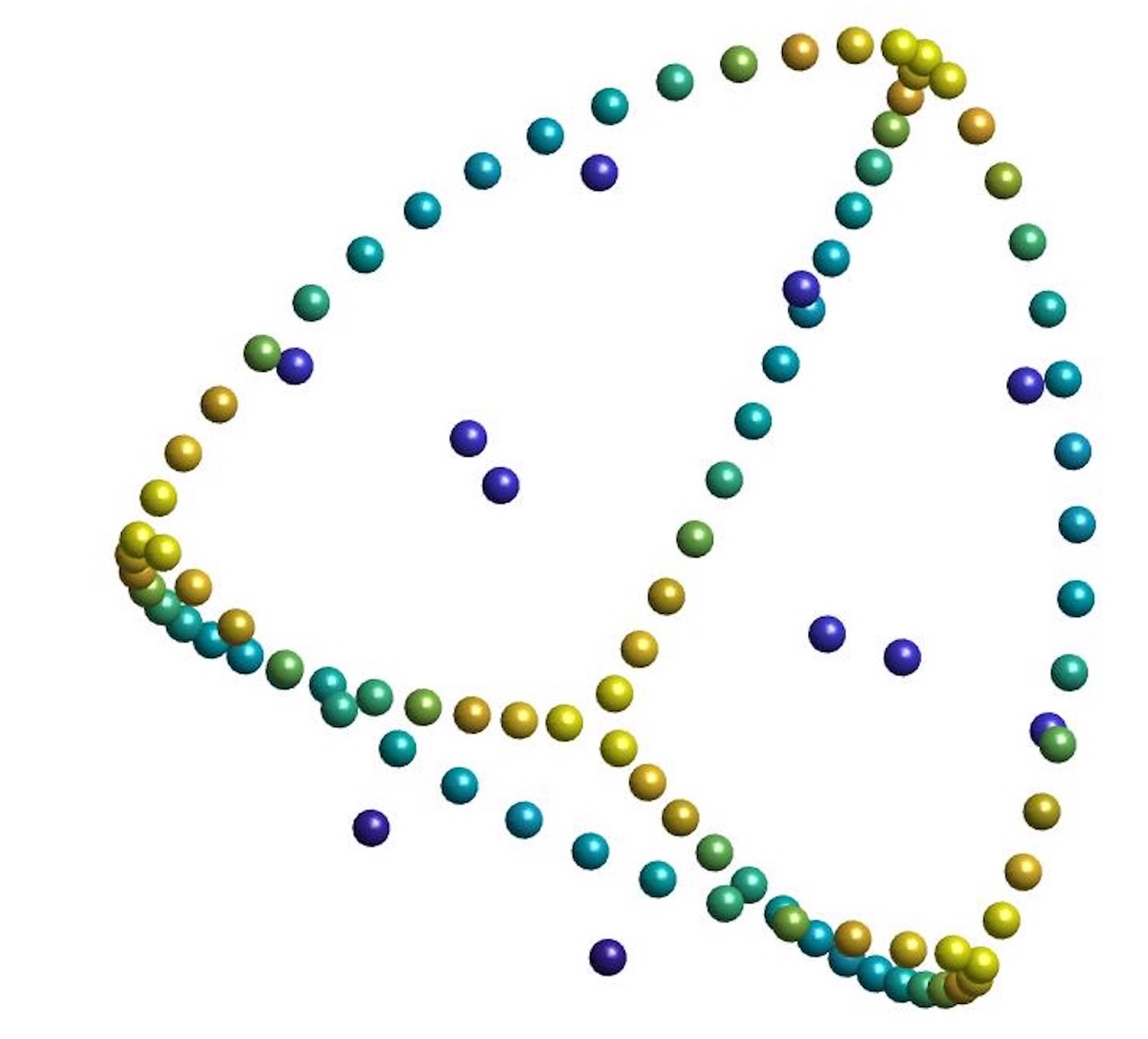}
  	\caption{Results of particle simulations 
in dimension $n=3$. The exponent of
attraction ranges through $\alpha = 2, 20, 200$ (from left to right).
The exponent of repulsion is $\lambda = 0.01$, which lies in 
the sub-Newtonian regime.}
\label{fig-3D}
\end{figure}

\section{Related work and  further questions} 

\subsection{Comparisons with Related Results}
According to Theorem~\ref{thm:support},
every minimizer $\mu$ of the 
functional $\E_{\infty,\lambda}$
is supported on a convex body $W_\mu$ of constant
width 1, and the following relations, summarized in the table below, hold true. 

\begin{table}[!h]
\begin{center}
\caption{\small Characteristics of minimizers of $\E_{\infty,\lambda}$ 
in terms of a body $W_\mu$ of constant width.}
\label{table1}
\begin{tabular}{| c | c |} 
 \hline
  Repulsion & Geometry\\ \hline 
 $\lambda<n-2$ &$ \supp\mu\subset\partial W_\mu$  \\ \hline 
 $\lambda=n-2$ &  $\supp\mu=\partial W_\mu$  \\ \hline 
 $\lambda>n-2 $& $ \supp\mu=W_\mu$ \\ \hline 
\end{tabular}
\end{center}
\end{table}

\noindent In particular, the Hausdorff dimension of $\mu$ satisfies
$$ 
{
\text{\rm dim} (\supp \mu)} \ 
\begin{cases} \le n-1\,, \qquad& \lambda\in (0, n-2)\,,\\
=n-1\,,& \lambda=n-2\,,\\
= n\,, & \lambda\in (n-2,n)\,.
\end{cases}
$$
To offer some perspective, note that 
classical results of geometric measure theory imply that 
{$\dim (\supp \mu)\ge \lambda$}  for
every Borel measure $\mu$ with 
$\mathcal{E}_{\infty,\lambda}(\mu) < \infty$
(see for example Theorem 4.13 in~\cite{F}).
For minimizers of energy functionals defined by
attractive-repulsive pair interaction kernels, 
a stronger lower bound was obtained in~\cite[Theorem 1]{Ba-etal}. 
Specifically, minimizers of $\E_{\alpha,\lambda}$
in the sub-Newtonian regime $\lambda\in (0,n-2)$
satisfy
\begin{equation}\label{carrillo}
\dim( \supp \mu)\ge \lambda+2\,.
\end{equation}
When $\lambda\in (n-3,n-2)$ this lower bound exceeds $n-1$, 
and in particular exceeds the dimension of
the support of the corresponding minimizer of 
$\E_{\infty,\lambda}$.
The results of~\cite{Ba-etal} apply more generally 
to {\em local minimizers}, in an optimal 
transport topology, for a larger class of 
attractive-repulsive functionals with integrable singularities
at the origin. In light of (\ref{carrillo}), which holds for all $\alpha >0$, the dimensional reduction of the support for $\lambda \le n-2$ (cf. Theorem \ref{thm:support} or Table \ref{table1}) is only achieved in the limit. 
In this limit,  the dimension of 
minimizers  are strictly smaller then 
those in the finite regime, a 
consequence of the strength of the diameter constraint.

Through Theorem~\ref{thm:limit} and Lemma~\ref{lem:isodiametric}
(below), the question of what the minimizers of the 
limiting functional look like is transformed into an 
isodiametric capacity problem: 
{\em For a given $\lambda\in (0,n)$, which sets of diameter 1 
have the largest $\lambda$-capacity?}
Although for any given set $W\subset\R^n$ the equilibrium measure
that realizes the capacity is unique, there could be more than one 
capacity-maximizing set.

One candidate for a set that maximizes
capacity among sets of diameter 1
is the ball of radius $\frac12$, 
which uniquely maximizes volume under the diameter restriction. 
For each $\lambda\in (n-2,n)$, the equilibrium measure on the ball 
is a well-known positive, radially symmetric 
density, and for $\lambda\le n-2$ it is the uniform measure on the 
boundary sphere~\cite[p. 163]{L}.
Note, however, that the 
ball {\em minimizes} capacity among sets of given volume,
indicating competition between size and shape
in the isodiametric problem.

There are a number of related results for 
the weak repulsion regime (corresponding to 
$\lambda<0$) which imply that the support of minimizers 
has dimension zero~\cite[Theorem 2]{Ba-etal}
provided that the pair interaction kernel vanishes 
of higher order as $|x-y|\to 0$.
In particular, the variance is maximized, among 
probability measures on $\R^n$ whose
support has diameter one, by the uniform
measure on the vertices of the unit simplex~\cite{LM}.

\subsection{Restricting Problem~\eqref{prob-P} to Densities and Sets}

In an interesting variant of Problem~\eqref{prob-P},
the minimization is restricted  to absolutely
continuous probability measures $\mu=\rho dx$ with
density bounded by $\rho\le m^{-1}$ for some
$m>0$.
\begin{equation}\label{prob-R}
\begin{cases} {\rm Minimize} \quad  
{\displaystyle
\E'_{\alpha,\lambda}(\rho)\, :=\, \int_{R^n}\int_{R^n}
K_{\alpha,\lambda}(x-y)\rho(x)\rho(y)\, dxdy}
\\[.5em]
{\rm over} \,\, \mathcal{A}_m:=
{\displaystyle
\bigl \{ \rho\in L^1(\R^n)\ \big\vert\ 0\le \rho\le m^{-1}
\,,\ \int_{\R^n} \rho\, dx=1 \bigr\}\,.}
\end{cases}
\end{equation} 
The density constraint plays the role of
an additional repulsive term in the energy.
This is relevant for biological aggregation problems, 
where the density of individuals cannot exceed a certain critical value.
By rescaling, Problem~\eqref{prob-R}
is equivalent to minimizing $\E'_{\alpha,\lambda}(\rho)$
among measures of mass $m$, subject to
the density constraint $\rho\le 1$. Unlike
{Problem~(\ref{prob-P})}, the mass $m$ does not scale out of the problem. 
It is known that for each $\alpha>0$ and $\lambda\in (0,n)$,
the functional $\E'_{\alpha,\lambda}$
has a minimizer on $\mathcal{A}_m$
for any $m>0$ (cf. \cite{CFT}). 

Since the set of probability measures of density at most $m^{-1}$
is a closed convex subspace of $\P$,
Theorems~\ref{thm:limit} and~\ref{thm:minimizer}
continue to hold.

\begin{corollary} [{Strong attraction limit with density
constraint}]
\label{cor:limit}
For $\lambda\in (0,n)$ and $\mu\in \P$,
let $\E'_{\alpha,\lambda}$ be as in Problem~\eqref{prob-R},
and define
$\E'_{\infty,\lambda}(\rho):=\E_{\infty,\lambda}(\rho\, dx)$
for $\rho\in L^1$.
Then 
\begin{enumerate}
\item 
$\E'_{\alpha,\lambda}\xrightarrow{\Gamma} \E'_{\infty,\lambda}$ 
as $\alpha \to \infty$  in the weak topology on $L^1$.
\item For each $m\le |B_{\frac12}|$, 
the functional $\E'_{\infty,\lambda}$
attains a global minimum on $\mathcal{A}_m$.
\end{enumerate}
\end{corollary} 
\noindent The assumption on $m$ guarantees that the
energy of the uniform measure on $B_{\frac12}$
remains bounded as $\alpha\to\infty$ 
{(}see the proof of 
Theorem~\ref{thm:minimizer}{)}.
{As $m\to 0$,} the measures corresponding to
a sequence of minimizers converge (up to translations, along
suitable subsequences, weakly in $\P$) to minimizers of $\E_{\infty,\lambda}$.

Problem~\ref{prob-R} is of interest also
when $m$ is large.  Under certain assumptions on $\lambda$ and $\alpha$,
$\E'_{\alpha,\lambda}$ is minimized 
for $m$ sufficiently large by the uniform 
probability density on a set $S$ of volume 
$m$ (\cite{BCT,FL,Lo}). In the context of aggregation
models, this indicates the formation of
a swarm.  A minimizing set is the  solution of 
the purely geometric, non-local
shape optimization problem
\begin{equation}\label{prob-S}
\begin{cases} {\rm Minimize} \quad  
\E''_{\alpha,\lambda}(S)\, := \E_{\alpha,\lambda} (\nu_S) \\[.5em]
{\rm over} \,\, \mathcal{S}_m:=
\bigl \{ S\subset\R^n\ \big\vert\ |S|=m\bigr\}\,,
\end{cases}\end{equation} 
where $\nu_S$ is the uniform probability measure on $S$.
It turns out that the infimum in Problem~\eqref{prob-S}
agrees with Problem~\eqref{prob-R}, but it is not always attained.
If the density of a minimizer
of $\E'_{\alpha,\lambda}$ on $\mathcal{A}_m$ falls 
strictly between $0$ and $m^{-1}$ on all or part 
of its support, then the shape optimization
problem~\eqref{prob-S} has no solution~\cite[Theorem 4.4]{BCT},
indicating a failure to fully aggregate.  
In this case, minimizing sequences for
Problem~\eqref{prob-S} diverge due to oscillations.
When $m$ is too small, typically $\rho<m^{-1}$
everywhere (cf. \cite{BCT, FL, Lo}), preventing
even partial aggregation.

All known solutions of the shape optimization problem~\eqref{prob-S}
are radially symmetric, and in many cases 
they are large balls~(cf. \cite{BCT,FL, Lo}). 
It may be possible to discover
interesting examples of symmetry-breaking 
in the strong-attraction limit, 
using Corollary~\ref{cor:limit} and the 
known relation between Problems~\eqref{prob-R} and~\eqref{prob-S}.

We are not aware of any explicit characterization
of the minimizers for
$\E'_{\infty,\lambda}$ on $\mathcal{A}_m$, 
even in the Newtonian case.  Suppose that $W$ maximizes 
capacity among sets of given diameter.
Since the density constraint prevents minimizers 
to concentrate on a lower-dimensional set, 
one may wonder whether a thin 
neighborhood of $\partial W$ might appear as a solution
to Problem~\eqref{prob-S}, and whether 
such a solution persists for sufficiently large finite
values of $\alpha$? When $W$ is not a ball,
this could give rise to symmetry-breaking in
Problems~\eqref{prob-R} and~\eqref{prob-S}.

\section{Convergence}
\label{sec-converge}

We begin by recalling  a few definitions. 
Given a topological space $X$, let
$(G_n)_n$ be a sequence of functions on $X$.
We say that $(G_n)$ {\bf Gamma-converges} 
to a function $G$ ($G_n \xrightarrow{\Gamma} G$) 
if the following two conditions hold for every $x\in X$:
\begin{itemize}	
\item {\em Lower bound inequality:}\  for all sequences $(x_n)_n\subset X$ 
such that $x_n\rightarrow x\in X,$
$$
\liminf\limits_{n\rightarrow\infty} G_n(x_n)\geq G(x)\,;
$$
\item {\em Upper bound inequality:}\ 
for all $x\in X$ there exists a sequence $(x_n)_n\subset X$ 
such that $x_n\rightarrow x$ and
$$
\limsup\limits_{n\rightarrow\infty} G_n(x_n)\leq G(x)\,.
$$
\end{itemize}
Gamma-convergence has many useful implications,
the most important of which is that
if $x_n$ minimizes $G_n$ over $X$, then every cluster point 
of the sequence $(x_n)$ minimizes $G$ over $X$ (cf.~\cite{Br}).

Given a sequence of measures $(\mu_n)_n\subset\P$,  we 
say $(\mu_n)_n$ {\bf converge weakly} 
to $\mu\in\P$ 
($\mu_n\rightharpoonup\mu$) if 
$$
\lim_{n\to\infty} \int\phi \, d\mu_n = \int\phi \,d\mu
$$
for every bounded continuous function $\phi$
on $\R^n$.  This induces the weak topology on~$\P$. 

\begin{proof} [Proof of Theorem \ref{thm:limit}]
Let $\mu\in \P$ be given. In the case where 
$\diam(\supp\mu)>1$, choose two points
$p,q\in \supp\mu$ with $|p-q|>1$. 
By continuity of the distance function, 
there exist open neighborhoods $U, V$ of
$p$ and $q$ such that $\dist(U, V)>1$. 
For any sequence of measures $(\mu_n)$ with
$\mu_n\rightharpoonup \mu$ in $\P$, we have
\begin{align*}
\E_{\alpha,\lambda}(\mu_n)
&=\int_{\R^n}\int_{\R^n}|x-y|^{\alpha}+|x-y|^{-\lambda}\,
d\mu_n(x)d\mu_n(y)\\
&\ge \bigl(\dist(U,V)\bigr)^\alpha  \mu_n(U)\mu_n (V)\,.
\end{align*}
Since $\liminf \mu_n(U) \ge  \mu(U)>0$ and likewise for $V$,
it follows that $\E_{\alpha_n,\lambda}(\mu_n)\to \infty$
along every sequence $(\alpha_n)$ with
$\alpha_n\to\infty$, verifying
simultaneously the lower and upper bound inequalities
for this case.

Otherwise, $\diam(\supp\mu)\leq1$.  
To see the lower bound inequality, 
let $(\mu_n)$ be a sequence in $\P$
that converges weakly to $\mu$, and let $t>0$.
For every $\alpha>0$, 
$$
\E_{\alpha,\lambda}(\mu_n) 
\ge \int_{\R^n}\int_{R^n} 
\min\{|x-y|^{-\lambda}, t\} \, d\mu_n(x)d\mu_n(y)\,.
$$
Since $\R^n\times\R^n$ 
is separable, the product measures $\mu_n\times\mu_n$
converge weakly to $\mu\times\mu$, and thus
for any sequence $(\alpha_n)$,
$$
\liminf\limits_{n\rightarrow\infty} 
\E_{\alpha_n,\lambda}(\mu_n) 
\ge \int_{\R^n}\int_{R^n} 
\min\{|x-y|^{-\lambda}, t\} \, d\mu(x)d\mu(y)\,.
$$
By monotone convergence, taking $t\to\infty$ 
yields the lower bound inequality.

The upper bound inequality is achieved by a sequence of 
properly chosen dilations of $\mu$. 
Given a sequence $\alpha_n\to\infty$, set
$\beta_n= e^{\frac{1}{\sqrt{\alpha_n}}}$, and define
a sequence of Borel measures by
$$
\mu_n(A)=\mu(\beta_n A)\,,\qquad n\ge 1\,.
$$
Since $\beta_n\to 1$, clearly $\mu_n\rightharpoonup \mu$.
We estimate
\begin{align*}
\E_{\alpha_n,\lambda}(\mu_n)
&=\int_{\R^n}\int_{\R^n}|x-y|^{\alpha_n}+
|x-y|^{-\lambda}\, d\mu_n(x)d\mu_n(y)\\
&=\int_{\R^n}\int_{\R^n}
\Bigl(\frac{|x-y)|}{\beta_n}\Bigr)^\alpha 
+ \Bigl(\frac{|x-y)|}{\beta_n}\Bigr)^{-\lambda}\, d\mu(x)d\mu(y)\\
&\le e^{-\sqrt{\alpha_n}} +
e^{\frac{\lambda}{\sqrt{\alpha_n}}}
\E_{\infty,\lambda}(\mu)\,.
\end{align*}
We have used that $|x-y|\le 1$
on the support of $\mu$ to bound the first summand
of the integrand,
and inserted the definition of the limiting functional
into the second summand. The desired inequality follows
upon taking $n\to\infty$.
\end{proof}

The proof of Theorem~\ref{thm:minimizer}
requires a compactness argument. To this end one often resorts to an application of 
 Lions'
concentration compactness principle
for probability measures
(cf.~\cite[Section 4.3]{St}) which asserts that 
every sequence $(\mu_n)_{n}$ in $\P$ has a subsequence
$(\mu_{n_k})_{k}$ satisfying one of the three following 
alternatives: (i) tightness up to translation (ii) vanishing (mass sent to infinity) or (iii) dichotomy (splitting). 
A standard technique it to show that (ii) and (iii) can not happen, yielding (i) which, precisely, means: 
 There exists a 
sequence $(y_k)_{k}\subset\R^n$ such that for all
$\epsilon>0$ there exists $R>0$ with the property that
$
\mu_{n_k}(B_R(y_k)) \geq 1-\epsilon$  for all $k$.


However, in our simpler case we may just as well directly prove 
tightness magentato obtain compactness.
\begin{lemma} 
\label{lem:compact}
Let $\E_{\alpha,\lambda}$ be as in Eq.~\eqref{prob-P},
let $(\alpha_n)$ be a sequence with $\alpha_n\to \infty$,
and fix $\lambda\in (0,n)$.  Then every sequence $(\mu_n)$ in $\P$ such 
that $\E_{\alpha_n,\lambda}(\mu_n)$ 
is bounded has a subsequence that converges 
weakly, up to translations, to some $\mu\in \P$.
\end{lemma}

\begin{proof} 
Let $(\mu_n)$ be such that
$$
\sup\limits_{n\in\mathbb{N}}\E_{\alpha_n,\lambda}(\mu_n)<\infty\,.
$$
Fix an $R>1$.  We have the lower bounds
\begin{align*}
\E_{\alpha_n,\lambda}(\mu_n)
& \ge \iint_{|x-y|\ge R} R^{\alpha_n} \, d\mu_n(x)d\mu_n(y)\\
&\ge R^{\alpha_n} \int_{\R^n} \mu_n\bigl(\R^n\setminus B_R(y)\bigr)
\, d\mu_n(y)\\
&\ge R^{\alpha_n} \bigl(1-\sup_{y\in\R^n}\mu_n(B_R(y)\bigr)\,.
\end{align*}
Since the left hand side is bounded by assumption 
while $\alpha_n\to\infty$, it follows that
$\sup_{y\in\R^n} \mu_n(B_R(y))\to 1$.
This establishes the first alternative of Lions'
concentration compactness principle.

Choose a sequence $(y_n)\subset\R^n$ such that 
$$\lim_{n\to\infty} \mu_n(B_2(y_n))= 1\,.
$$
Given $\eps>0$, let $N$ be so large that $\mu_n(B_2(y_n))\ge 1-\eps$ 
for all $n> N$.  Then choose $R$ so large that 
$\mu_n(B_R(y_n))\ge 1-\eps$ for $n=1,\dots,N$.
Taking taking $R\ge 2$ ensures that
$\mu_n(B_R(y_n))\ge \mu_n(B_2(y_n))\ge 1-\eps$
also for $n>N$. 

Let $(\tilde \mu_n)_n$ be the
sequence of translates of  $\mu_n$ defined by
$$
\tilde \mu_n(A) = \mu_n (y_n+A)\,,\quad n\ge 1
$$
for each Borel set $A\subset\R^n$. 
Since $(\tilde \mu_n)$ is tight. Prokhorov's theorem
yields a subsequence $(\tilde \mu_{n_k})_k$
that converges weakly in $\P$.
 \end{proof}

\begin{proof}[Proof of Theorem~\ref{thm:minimizer}]
Let $(\alpha_n)$ be a nonnegative
sequence with $\alpha_n\to\infty$,
and let $(\mu_n)$ be a sequence of measures such that 
each $\mu_n$ minimizes $\E_{\alpha_n,\lambda}$. 
We will prove that
$(\E_{\alpha_n,\lambda}(\mu_{\alpha_n}))_n$ 
is bounded, and then apply Lemma~\ref{lem:compact}.

Let $\nu$ be the uniform probability measure
on the ball of radius $\frac12$.
Since $\mu_n$ minimizes
$\E_{\alpha_n,\lambda}$ for each $n$, we have
\begin{align*}
\E_{\alpha_n,\lambda}(\mu_n) 
&\le \E_{\alpha_n,\lambda}(\nu) \\
&=\int_{\R^n} \int_{\R^n} |x-y|^\alpha + |x-y|^{-\lambda}\, d\nu(x)d\nu(y)\\
&\le 1+ \int_{\R^n}\int_{\R^n} |x-y|^{-\lambda}\, d\nu(x)d\nu(y)\\
&<\infty\,.
\end{align*}
In the last two inequalities, we have used that
the support of $\nu$ has diameter one, and that the kernel is locally 
integrable.  

By Lemma~\ref{lem:compact} there exists a subsequence 
$\mu_{n_k}$ that converges weakly up to translation,
to some measure $\mu\in \P$.
Since the functionals are translation
invariant, we may assume that the sequence of
minimizers itself that has a subsequence converging weakly
to $\mu$. By the properties of the Gamma-limit,
$\mu$ is a global minimizer of $\E_{\infty,\lambda}$. \end{proof}

\section{Characterization of Minimizers}\label{sec-characterization}

We recall some classical results from potential theory.  First, recall the  {\it $\lambda$-capacity} of a set $A\subset\R^n$ previously 
defined in  ({\ref{lam-cap}) as the reciprocal
of the minimum of the repulsive energy $I_\lambda$
over measures supported in $A$.}
If $A$ is a compact set of positive Lebesgue measure,
the $\lambda$-capacity is finite by the local integrability of
the Riesz-potential, and the
supremum is achieved by some measure $\mu\in\P$. Since
$I_\lambda$ is positive definite,
the minimizer is unique.

The next lemma relates the minimization problem
for $\E_{\infty,\lambda}$ to an isodiametric
capacity problem.

\begin{lemma} 
\label{lem:isodiametric}
Let $n\ge 1$, $\lambda\in (0,n)$.  Then
$$
\inf_{\nu\in \P} \E_{\infty,\lambda} (\nu)
= \Bigl( \sup_{A\subset\R^n}\bigl\{ C_\lambda(A)\ \big\vert\ 
\diam (A) \le 1\bigr\}\Bigr)^{-1}\,.
$$
Furthermore, the infimum on the left hand
side is attained for some measure $\mu$ with
$\diam(\supp\mu)=1$, and the supremum on the right
hand side is attained for some convex body $W\subset\R^n$
of constant width 1 containing the support of $\mu$.
Conversely, if $W$ maximizes $\lambda$-capacity among bodies of
constant width, then the equilibrium measure
on $W$ attains the minimum on the left hand side.
\end{lemma}

\begin{proof}
We split the minimization problem for
$\E_{\infty,\lambda}$ into two steps,
\begin{align*}
\inf_{\nu\in \P} \E_{\infty,\lambda} (\nu)
&=\inf_{A\subset\R^n}
\Bigl\{\inf_{\nu\in \P}\bigl\{ I_\lambda(\nu)\ \big\vert\ 
\supp\nu\subset A\bigr\}\ \Big\vert\ \diam(A)\le 1 \Bigr\}\\
&=\Bigl(\sup_{A\subset\R^n} 
\bigl\{ C_\lambda(A)\ \big\vert\ \diam(A)=1\bigr\}\Bigr)^{-1}\,.
\end{align*}
By Theorem~\ref{thm:minimizer}, 
the infimum on the left hand side is
attained for some measure $\mu\in\P$.
Clearly, $\diam(\supp \mu)=1$, since otherwise
$\mu$ could be rescaled to lower the value
of $\E_{\infty,\lambda}$.
Moreover, $A=\supp\mu$ achieves the
supremum on the right hand
side, and $\mu$ is the equilibrium measure for
the capacity $C_\lambda(A)$.
Since the capacity increases monotonically under inclusion,
we may replace $A$ by its convex hull.
The last claim follows since every closed convex set
of diameter~1 is contained in 
a convex body $W$ of constant width 1 (cf.  \cite{S}).
Since $C_\lambda(W)=C_\lambda(\supp \mu)$,
if follows that $\mu$ is the equilibrium measure also for
$W$.
\end{proof}

We can now appeal to known properties of equilibrium 
measures in classical potential theory. 
Given a probability measure $\mu$ on $\R^n$
and $\lambda\in (0,n)$, we define the corresponding
potential by
$$
\phi^\mu_\lambda(x)\, : =\, \int_{\R^n}  |x-y|^{-\lambda}\,d\mu(y)\,.  
$$
{
For any $x\in\R^n$, the integral is well-defined and 
strictly positive, though possibly infinite.
The function} has the following regularity property 
outside the support of~$\mu$. 

\begin{lemma}\label{lem:subharmonic} Let $\mu$ be a probability measure 
on $\R^n$.  On $\R^n\setminus \supp \mu$,
the potential $\phi_\mu^\lambda$ is smooth and 
$$
\begin{cases}
\text{strictly subharmonic}\,\qquad & \lambda\in (0,n-2)\,,\\
\text{harmonic}\,\qquad & \lambda= n-2\,,\\
\text{strictly superharmonic}\,\qquad & \lambda\in (n-2,n)\,.
\end{cases}
$$
\end{lemma}
\begin{proof} By direct computation,
$$
\Delta\phi^\mu_\lambda(x)\, =\, \lambda(\lambda+2-n)\int_{\R^n} 
|x-y|^{-\lambda-2}\, d\mu(y) 
$$
away from the support of $\mu$.
\end{proof}

In the super-Newtonian regime,
the equilibrium measure has the following property.

\begin{lemma} {\rm \cite[p.137]{L}}
\label{lem:potential}
Let $\lambda\ge n-2$, and let $W\subset\R^n$ be a compact
set of positive capacity. If 
$\mu\in\P$ minimize $I_\lambda$ among
probability measures supported 
on $W$, then
\begin{align*}\phi^\mu_\lambda(x)=I_\lambda(\mu)&\ \ \ 
\text{approximately everywhere on }W\\
\phi^\mu_\lambda(x)\leq I_\lambda(\mu)&\ \ \ \text{throughout }\R^n
\end{align*}
where approximately everywhere means everywhere 
except on a set of capacity zero.
\end{lemma} 

We are ready for the proof of Theorem~\ref{thm:support}.

\begin{proof} 
Let $\mu$ be a minimizer of $\mathcal{E}_{\infty,\lambda}$. 
By Lemma~\ref{lem:isodiametric},
$\mu$ is the equilibrium measure that achieves
the $\lambda$-capacity of some convex body $W$ of constant width 1.  
When $\lambda\leq n-2$, 
classical results of potential theory (cf. \cite[p.162]{L}) ensure
that $\supp\mu\subset \partial W$. 
This proves the claim in the {sub-Newtonian} regime.

Let now $\lambda\geq n-2$, and $p\in\partial W$.
Since $W$ is a convex body, every neighborhood of
$p$ intersects the interior of $W$ in a set of positive volume
({and} hence positive capacity). Again by classical results
of potential theory (cf. \cite[p.164]{L}), 
$p$ lies in the support of~$\mu$. Therefore
$\partial W \subset \supp\mu$. 
{Together with the result for $\lambda\le n-2$,
this completes the proof in the Newtonian case}.

{
For $\lambda>n-2$ Lemma~\ref{lem:subharmonic}
yields that the potential} $\phi_\mu^\lambda$ is strictly subharmonic 
outside the support of~$\mu$. By the {strong}
maximum
principle, $\phi_\mu^\lambda$ is non-constant on every
non-empty open set $U$ with $\mu(U)=0$.  On the other hand,
$\phi_\mu^\lambda$ is constant on the interior of
$W$ by Lemma~\ref{lem:potential}. Therefore
$\mu(U)>0$ for every non-empty open subset of the interior of $W$,
{and we conclude that $W\subset \supp \mu$.
This proves the claim in the super-Newtonian regime.}
 \end{proof} 

\section{Capacity Estimates}\label{sec-cap-est}

We close with some simple capacity
estimates 
which will prove Theorem~\ref{thm:sym-break}.


\begin{lemma}
\label{lem:capacity-bound} 
Let $n\ge 1$, $\lambda\in (0,n)$. Then
$$
\sup_{A\subset\R^n} 
\bigl\{ C_\lambda(A)\ \big\vert\ \diam(A)=1\bigr\} < 1\,.
$$
\end{lemma}

\begin{proof} 
By Lemma~\ref{lem:isodiametric}, there is a set $A\subset\R^n$
that maximizes the capacity $C_\lambda$
among sets of diameter~1. Let
$\mu$ be the equilibrium measure on $A$ that achieves 
the capacity.  We estimate
$$
\mathcal{E}_{\infty,\lambda}(\mu) -1 
\ge \int
\Bigl(\int_{B_{\frac12}(x)}
(|x-y|^{-\lambda}-1)\, d\mu(y)\Bigr)d\mu(x)>0\,,
$$
where the first inequality holds since
the integrand is nonnegative for every pair of
points $x,y\in A$, and the second inequality
uses that $\mu (B_{\frac12}(x))>0$ for 
$x$ in the support of $\mu$.
By Lemma~\ref{lem:isodiametric}, 
$C_\lambda (A)= (\mathcal{E}_{\infty,\lambda}(\mu))^{-1}<1$, as claimed.
\end{proof}

We next consider the capacity of balls in high dimensions.

\begin{lemma}
\label{lem:ball-cap}
 For every $\lambda>0$
$$\lim\limits_{n\rightarrow\infty} C_\lambda(B^{(n)}_{1/2})=2^{-\frac{\lambda}{2}}.$$
\end{lemma} 
\begin{proof}
 This follows by direct computation of $C_\lambda(B^{(n)}_{1/2})$ (cf. \cite[p.163]{L}) and Stirling's approximation.
\end{proof}

Finally, we construct sets of larger capacity in high dimensions.

\begin{lemma}
\label{lem:large-n}
For every $\lambda>0$,
$$
\lim_{n\to\infty}
\left(
\sup \left\{ C_\lambda(A) \ \big\vert\ 
A\subset \R^n, \diam(A)\le 1\right\}\right)  = 1\,.
$$
\end{lemma}

\begin{proof}
Since $C_\lambda(A)<1$ for all $n$ by Lemma~\ref{lem:capacity-bound},
it suffices to establish the corresponding
lower bound on the capacity.

We will construct a family of subsets $(A_n)_n$ of diameter~1
in the sphere of radius $\frac12\sqrt{2}$
in $\R^n$ that achieves this limit. For each $n>\lambda+1$, 
the spherical cap of diameter~1 in this sphere
has positive $\lambda$-capacity. Let $A_n$ be a set of maximal
capacity among such subsets,
and let $\mu_n$ be the equilibrium measure on $A_n$
that attains the capacity.

For $m,n>\lambda+1$, consider a convex combination
$$
\mu=(1-t)(\mu_m\otimes \delta) + t (\delta \otimes \mu_n)
$$
on $\R^{m+n}$, where $\delta$ denotes the unit mass at $0$ in
$\R^n$ and $\R^m$, respectively, 
and $t\in(0,1)$ will be chosen below.
By definition, $\mu$ is supported on
$(A_m\times \{0\}) \cup (\{0\}\times A_n)$,
which lies in the sphere of radius $\frac12 \sqrt{2}$
in $\R^{m+n}$ and has diameter 1. We estimate
\begin{align*}
\mathcal{E}_{\infty,\lambda}(\mu_{m+n})-1
&\le \mathcal{E}_{\infty,\lambda}(\mu)-1\\
&=\int\int (|x-y|^{-\lambda}-1)\, d\mu(x)d\mu(y)\\
&= (1-t)^2 (\mathcal{E}_{\infty,\lambda}(\mu_m)-1)
+ t^2 ( \mathcal{E}_{\infty,\lambda}(\mu_n)-1)\,;
\end{align*}
the mixed terms vanish because 
$|x-y|=1$ whenever $x\in A_m\times\{0\}$ and $y\in \{0\}\times A_n$. 
Minimization over $t$ yields
\begin{align*}
\mathcal{E}_{\infty,\lambda}(\mu_{m+n}) -1 
&\le \frac{(\mathcal{E}_{\infty,\lambda}(\mu_m)-1)
(\mathcal{E}_{\infty,\lambda}(\mu_n)-1)}
{\mathcal{E}_{\infty,\lambda}(\mu_m) + \mathcal{E}_{\infty,\lambda}(\mu_n)
-2}\,.
\end{align*}
Since $\mathcal{E}_{\infty,\lambda}(\mu_n)>1$ for all $n$
by Lemma~\ref{lem:capacity-bound}, we
can pass to reciprocals and conclude that
$\left(\mathcal{E}_{\infty,\lambda}(\mu_{n})-1\right)^{-1}$
is superadditive in $n$. 
By Fekete's superadditivity lemma
$$
\lim_{n\to\infty} \frac1n \left(\mathcal{E}_{\infty,\lambda}(\mu_{n})-1\right)^{-1}
= \sup_n \frac1n \left(\mathcal{E}_{\infty,\lambda}
(\mu_{n})-1\right)^{-1} >0\,.
$$
It follows that
$\displaystyle{\lim_{n\to\infty} C_\lambda(A_n) = \left(
\lim_{n\to\infty} \mathcal{E}_{\infty,\lambda}(\mu_n)\right)^{-1}
= 1}$.
\end{proof}

The proof of Theorem~\ref{thm:sym-break} is an immediate corollary of Lemma~\ref{lem:ball-cap} and Lemma~\ref{lem:large-n} 
since $2^{-\frac{\lambda}{2}}<1$ for every $\lambda>0$. Note that 
the near-maximizers constructed in the proof of 
Lemma~\ref{lem:large-n} have dimension much below $n$, but
this need  not be true for actual maximizers.

\subsection*{Acknowledgments.} We  thank the anonymous referee for comments which helped improve the paper. This work emerged from an 
NSERC  URSA (Undergraduate Research Summer Award) project of 
Elias Hess-Childs supervised by AB and RC. 
AB and RC also acknowledge the support of NSERC 
through its Discovery Grants program. The authors 
thank Ihsan Topaloglu 
for many discussions and comments. We also 
thank Theodore Kolokolnikov for providing us with 
his 3D particle simulation code (cf. Fig~\ref{fig-3D}).

\bibliographystyle{amsalpha}

\begin{thebibliography}{21}

\bibitem{Ba-etal}
D. Balagu{\'e}, J.A. Carrillo,  T. Laurent,  and  G. Raoul, 
\textit{Dimensionality of local minimizers of the interaction energy}, 
Arch. Rat. Mech. Anal., Vol. 209 (2013), 1055-1088. 


\bibitem{BT} 
A.J. Bernoff and C. Topaz,  \textit{Nonlocal Aggregation Models: 
A Primer of Swarm Equilibria}, SIAM Rev., 55(4)  (2013),  709-747.



\bibitem{Br} A. Braides: $\Gamma$-convergence for beginners. Oxford University Press, (2002).

\bibitem{BCT} A. Burchard, R. Choksi, and I. Topaloglu, \textit{Nonlocal shape optimization via interactions of attractive and repulsive potentials}, Indiana Math. J., Vol. 67-1 (2018), 375-395. 


\bibitem{CaDiFiLaSl} J.A. Carrillo, M. DiFrancesco, A.  Figalli, T. 
              Laurent, and D. Slep{\v{c}}ev, \textit{Global-in-time weak measure solutions and finite-time
              aggregation for nonlocal interaction equations},
 Duke Math. J. Vol 156-2 (2011), 229-271.
  
\bibitem{CCP}
J. A. Canizo, J. A. Carrillo, and F. S. Patacchini, \textit{ Existence of compactly supported global minimizers for the interaction energy}, Arch. Ration. Mech. Anal. Vol. 217-3, (2015), 1197-1217.

\bibitem{CFT} R. Choksi, R.C. Fetecau, and I. Topaloglu, \textit{On minimizers of interaction functionals with competing attractive and repulsive potentials}. Ann. Inst. H. Poincar\'e Anal. Non Lin\'eaire, Vol. 32,  (2015), 1283--1305.

\bibitem{F} K. J. Falconer. The Geometry of Fractal Sets. Cambridge 
Tracts in Mathematics Cambridge University Press, Cambridge, 1986.


\bibitem{FL}
R.~L. Frank and E.~H. Lieb. \textit{A `liquid-solid' phase transition 
in a simple model for swarming based on the `no flat-spots' 
theorem for subharmonic functions}. Indiana Univ. Math. J.,  
Vol. 67-4 (2018).


\bibitem{L} N. S. Landkof. Foundations of Modern Potential Theory. Springer-Verlag, Berlin, 1972.


\bibitem{LL} E. H. Lieb and M. Loss, Analysis. Second edition. 
Graduate Studies in Mathematics \textbf{14}, American Mathematical 
Society, Providence, RI, 2001. 

\bibitem{LM}
T. Lim and R. J. McCann, \textit{Isodiametry, variance, and regular simplices from particle interactions}, arXiv:1907.13593 (2019). 

\bibitem{Lo}
O. Lopes, \textit{Uniqueness and radial symmetry of minimizers for a nonlocal variational problem},
Communications on Pure and Applied Analysis 18-5 (2019),  2265-2282.

\bibitem{S} P. R. Scott, \textit{Sets of Constant Width and Inequalities}, The Quarterly Journal of Mathematics, Volume 32, Issue 3, 1 September 1981, 345-348

\bibitem{SST} R. Simione, D. Slep{\v{c}}ev, and I. Topaloglu, \textit{Existence of ground states of nonlocal-interaction energies}. J. Stat. Phys. \textbf{159} (2015), no. 4, 972--986.

\bibitem{St}
M. Struwe,
  Variational Methods: Applications to Nonlinear Partial Differential Equations and
              Hamiltonian Systems,
Third Edition, Springer-Verlag, 2000. 

\end{thebibliography}

\end{document}